\documentclass[a4paper]{amsart}
\usepackage[utf8]{inputenc}
\usepackage[T2A]{fontenc}
\usepackage[english]{babel}

\usepackage{amssymb}

\newtheorem{lemma}{Lemma}
\newtheorem{theorem}{Theorem}

\def\hm#1{#1\nobreak\discretionary{}{\hbox{\ensuremath{#1}}}{}}

\def\N{{\mathbb N}}
\def\CC{{\mathbb C}}
\def\B{{\mathcal B}}
\def\E{{\mathcal E}}

\def\taur{\tau^{(r)}}
\def\sigmar{\sigma^{(r)}}

\def\taues{\tau^{(e)s}}
\def\sigmaes{\sigma^{(e)s}}

\def\eps{\varepsilon}
\def\ups{\upsilon}
\def\res{\mathop{\mathrm{res}}}
\def\llog{\mathop{\mathrm{llog}}}
\def\lllog{\mathop{\mathrm{lllog}}}

\def\suma{\mathop{\sum\nolimits^*}}

\def\le{\leqslant}
\def\ge{\geqslant}
\def\proof{\par {\em Proof.}\hspace{1em}}
\def\endproof{{\hfil$\blacksquare$\parfillskip0pt\par\medskip}}

\usepackage[unicode]{hyperref}

\begin{document}

\title{Functions concerned with divisors of order $r$}
\author{Andrew V. Lelechenko}
\address{I.~I.~Mechnikov Odessa National University}
\email{1@dxdy.ru}


\begin{abstract}
N. Minculete has introduced a concept of divisors of order $r$: integer $d=p_1^{b_1}\cdots p_k^{b_k} $ is called a divisor of order $r$ of~$n=p_1^{a_1}\cdots p_k^{a_k}$ if $d \mid n$  and~$b_j\in\{r, a_j\}$ for~$j=1,\ldots,k$. One can consider respective divisor function~$\taur$  and sum-of-divisors function $\sigmar$.

In the present paper we investigate the asymptotic behaviour of $$\sum_{n\le x} \taur(n) \text{ and } \sum_{n\le x}\sigmar(n)$$
and improve several results of \cite{minculete2012th} and \cite{minculete2012b}. We also provide conditional estimates under Riemann hypothesis.
\end{abstract}

\maketitle

\section{Introduction}

Recently N. Minculete in his PhD Thesis \cite{minculete2012th}, devoted to the functions using exponential divisors, and in further paper \cite{minculete2012b} introduced a concept of {\em divisors of order $r$:} integer $d=p_1^{b_1}\cdots p_k^{b_k} $ is called a divisor of order $r$ of number $n=p_1^{a_1}\cdots p_k^{a_k}$ if $d$ divides $n$ in the usual sense and~$b_j\in\{r, a_j\}$ for~$j=1,\ldots,k$. We also suppose that~$1$ is a divisor of any order of itself (but not of any other number). Let us denote respective divisor and sum-of-divisor functions as $\taur$ and $\sigmar$. These functions are multiplicative and
\begin{align}
\label{eq:taur-def}
\taur(p^a) &= \begin{cases}
               1, & a\le r, \\
               2, & a >  r.
               \end{cases}
\\
\label{eq:sigmar-def}
\sigmar(p^a) &= \begin{cases}
               p^a, & a\le r, \\
               p^a+p^r, & a >  r.
               \end{cases}
\end{align}

In a special case of $r=0$ we get well-studied {\em unitary divisors}. For example,
it was proved in \cite{gioia1966} that
\begin{equation}\label{eq:gioia1966}
\sum_{n\le x} \tau^{(0)}(n) = {x\over \zeta(2)} \left( \log x + 2\gamma-1 - {2\zeta'(2)\over\zeta(2)} \right) + O(x^{1/2}).
\end{equation}
(under Riemann hypothesis error term is $O(x^{221/608+\eps})$ due to~\cite{kaczorowski2009})
and in \cite{sitaramachandrarao1973} it was proved that
\begin{equation}\label{eq:sitaramachandrarao1973}
\sum_{n\le x} \sigma^{(0)}(n) = {\pi^2 x^2 \over 12 \zeta(3)} + O(x \log^{5/3} x).
\end{equation}

In another special case of $r=1$ we get so-called by Minculete {\em exponential semiproper divisors} and denote  $\taues:=\tau^{(1)}$, $\sigmaes:=\sigma^{(1)}$. An integer $d$ is an exponential semiproper divisor of $n$ if~$\ker d = \ker n$ and~$(d/\ker n, n/d) = 1$, where~$\ker n = \prod_{p \mid n} p$.

Minculete proved in \cite[(3.1.17--19)]{minculete2012th} that
\begin{gather}
\label{eq:sup-taur}
\limsup_{n\to \infty} {\log \taur(n) \log\log n \over \log n } = {\log 2 \over r+1},
\\
\label{eq:average-minc}
\sum_{n\le x} \taur(n) = {\zeta(r+1) \over \zeta(2r+2)} x + Ax^{1/(r+1)} + O(x^{1/(r+2) + \eps}),
\\
\label{eq:sup-sigmar}
\limsup_{n\to \infty} {\sigmar(n) \over n \log\log n } = {6 e^\gamma \over \pi^2}.
\end{gather}

In the present paper we improve the error term in \eqref{eq:average-minc} and establish asymptotic formulas for~$\sum_{n\le x} \sigmar(n)$ with $O$- and $\Omega$-estimates of the error term.

\section{Notation}

In asymptotic relations we use $\sim$, $\asymp$, Landau symbols $O$ and $o$, big omegas $\Omega$ and $\Omega_\pm$, Vinogradov symbols $\ll$ and $\gg$ in their usual meanings. All asymptotic relations are given as an argument tends to the infinity.

Letter $p$ with or without indexes denote rational prime.

As usual $\zeta(s)$ is the Riemann zeta-function. For complex $s$ we denote~$\sigma:=\Re s$ and~$t:=\Im s$.

We use abbreviations $\llog x := \log\log x$, $\lllog x := \log\log\log x$.

Letter $\gamma$ denotes Euler–Mascheroni constant, $\gamma \approx 0.577$.

Everywhere $\eps>0$ is an arbitrarily small number (not always the same even in one equation).

We write $f\star g$ for Dirichlet convolution:
$ (f \star g)(n) = \sum_{d \mid n} f(d) g(n/d) $.

Function $\ker\colon \N\to\N$ stands for
$\ker n = \prod_{p \mid n} p$.

For a set $A$ notation $\#A$ means the cardinality of $A$.

\section{Preliminary estimates}

Consider
$$ \tau(a,b;n) = \sum_{k^a l^b=n} 1, \qquad T(a,b;x) = \sum_{n\le x} \tau(a,b;n), \qquad 1\le a\le b. $$
One can directly check that
\begin{equation*}\label{tauab-series}
\sum_{n=1}^\infty {\tau(a,b;n) \over n^s} = \zeta(as) \zeta(bs), \qquad \sigma > 1
\end{equation*}

\begin{lemma}
\begin{equation*}
T(a,b;x) = H(a,b;x) + \Delta(a,b;x)
\end{equation*}
where
$$ H(a,b;x) = \begin{cases}
              \zeta(b/a) x^{1/a} + \zeta(a/b) x^{1/b}, & 1\le a < b, \\
              x^{1/a} \log x + (2\gamma-1) x^{1/a}, & a=b,
              \end{cases}
$$
and
\begin{equation*}\label{eq:delta-simple-estimates}
x^{1/2(a+b)} \ll  \Delta(a,b;x) \ll \begin{cases}
                     x^{1/(2a+b)} & 1\le a < b, \\
                     x^{1/3a} \log x & a=b.
                     \end{cases}
\end{equation*}
\end{lemma}

\proof
See \cite[Th. 5.1, Th. 5.3, Th. 5.8]{kratzel1988}.
\endproof

In fact $\Delta(a,b;x)$ can be estimated more precisely. For our goals we are primarily interested in the behaviour of $\Delta(1,b;x)$. Let us suppose that
\begin{equation}
\label{eq:Delta-theta-theta'}
\Delta(1,b;x) \ll x^{\theta_b} \log^{\theta'_b} x,
\end{equation}
then due to \cite[Th. 5.11]{kratzel1988} we can choose
$$ \theta_b = {1 \over b+7/2}, \qquad \theta'_b = 1, \qquad b\ge 7. $$
Estimates for $b \le 16$ are given in Table \ref{tbl:tauab}. Estimate for $b=1$ belongs to Huxley~\cite{huxley2005}, and estimate for $b=2$ belongs to Graham and Kolesnik~\cite{graham1988}. We have found no references on the best known results for $b\ge3$, so we calculated them with the use of \cite[Th. 5.11, Th. 5.12]{kratzel1988} selecting appropriate exponent pairs carefully. It seems that some of this estimates may be new.

\begin{table}[b]
\label{page:tauab}
\begin{tabular}{@{}r|r|c|c@{}}
$b$ & $\theta_b$ & $\theta'_b$ & Exponent pair or reference\\\hline

1
& $131/416+\eps                  \approx 0.314904$
& 0
& \cite{huxley2005} \\\hline

2
& $1057/4785+\eps                \approx 0.220899$
& 0
& \cite{graham1988} \\\hline

3
& $1486/8647+\eps                \approx 0.171852$
& 0
& $AB(A\B)^2H$ \\ \hline

4
& $1448/10331+\eps               \approx 0.140161$
& 0
& $AH$ \\ \hline

5
& $(15921-2 M)/30437 \approx 0.121398$
& 2
& $A^2 \B A (A \B A \B^2 A^2 \B^6)^\infty I$ \\ \hline

6
& $669/6305                      \approx 0.106106$
& 1
& $(A^2B)^3(AB)^3A^4BI$ \\ \hline

7
& $(9370-M)/34469 \approx 0.094491$
& 2
& $A^2 \B^2 (\B A \B^2 A^2 \B^6 A)^\infty I$ \\ \hline

8
& $(5+\sqrt{809}) / 392 \approx 0.085314$
& 1
& $A^2 \B^4 A (A \B A^2)^\infty I$ \\ \hline

9
& $(10551-M)/56976 \approx  0.077892$
& 2
& $A^2 \B^2 (A \B^6 A \B A \B^2 A)^\infty I$ \\ \hline

10
& $150509/2096993+\eps           \approx 0.071774$
& 0
& $(A^2\B^2)^3A{\B}H$ \\ \hline

11
& $1048/15811+\eps               \approx 0.066283$
& 0
& $A^2H$ \\ \hline

12
& $64/1037+\eps                  \approx 0.061716$
& 0
& $A^2H$ \\\hline

13
& $\vphantom{\biggr(} {\displaystyle 2516635/43324033+\eps \atop        \hfill\displaystyle\approx 0.058089}$
& 0
& $A^3BA^3BA^2BA^4B(A\B)^2H$ \\ \hline

14
& $75/1373  \approx 0.054625$
& 1
& $A^2(A\B)^2BA^3BI$ \\ \hline

15
& $\vphantom{\biggr(} {\displaystyle 13514730527/262064292044+\eps \atop        \hfill\displaystyle\approx 0.051570}$
& 0
& $A(A^2B)^3A^4\B^7A^3{\B}BA^4{\B}H$ \\ \hline

16
& $15/307   \approx 0.048860$
& 1
& $A^3BA^2BA^4BI$ \\

\end{tabular}
\smallskip
\caption{Values of $\theta_b$ and $\theta'_b$ in \eqref{eq:Delta-theta-theta'} for $b\le 16$. Exponent pairs are written in terms of  $A$- and $B$-processes \cite[Th.~2.12, 2.13]{kratzel1988}. We abbreviate~$\B := BA$. Here~$I=(0,1)$ and $H=(32/205\hm+\eps, 269/410+\eps)$ is Huxley exponent pair from~\cite{huxley2005}. Also $M\hm=\sqrt{37368753}$.}
\label{tbl:tauab}
\end{table}

\medskip

\begin{lemma}\label{l:petermann1997}
Let $\alpha$ and $\beta$ be positive real numbers with $\beta+1 \le \alpha $.
Then
$$ \sum_{mn^\alpha \le x} mn^\beta = {\zeta(2\alpha-\beta) \over 2} x^2 + \mathcal{D}(\alpha,\beta; x), \quad
\mathcal{D}(\alpha,\beta; x) \ll \begin{cases}
                        x\log^{2/3}x, & \beta+1 = \alpha, \\
                        x,            & \beta+1 < \alpha.
                        \end{cases}
$$
\end{lemma}

\proof
See \cite[Th. 1]{petermann1997}.
\endproof

\medskip

For $k>0$ one can define a multiplicative function $\mu_k$  implicitly by
$$
\sum_{n=1}^\infty {\mu_k(n) \over n^s} = {1 \over \zeta(ks)} , \qquad \sigma > 1.
$$
So $\mu_k(n^k)=\mu(n)$ and $\mu_k(m)=0$ for all other arguments. Trivially~$\mu_1 \equiv \mu$. Then
\begin{equation*}\label{eq:sum-of-muk-asymptotic}
M_k(x) := \sum_{n\le x} \mu_k(n)
= \sum_{n\le x^{1/k}} \mu(n)
\ll x^{1/k} \exp\bigl(-CN(x)\bigr),
\end{equation*}
where $C>0$, $N(x) = \log^{3/5} x \llog^{-1/5} x $. See \cite[Th. 12.7]{ivic2003} for the proof of the last estimate. Assuming Riemann hypothesis (RH) we get much better result
$$ M_k(x) \ll x^{1/2k + \eps} \qquad \text{\cite[Th. 14.25 (C)]{titchmarsh1986}}.$$

\medskip

\begin{lemma}\label{l:kuhleitner1994}
Let $K\in\N$, $J\in \N \cup \{0\}$, $m_1\le\cdots\le m_K$, $n_1\le\cdots\le n_J$, where all~$m_k, n_j \in \N$, and suppose that
$$ \sum_{n=1}^\infty {a(n)\over n^s} =
{\zeta(m_1s)\cdots\zeta(m_Ks) \over \zeta(n_1s)\cdots\zeta(n_Js)}. $$
Let
$$ \alpha = {K-1 \over 2\sum_{k=1}^K m_k}. $$
If $1/\alpha < 2n_j $ for all $j=1,\ldots,J$ then for arbitrary $H(x)$ of the form
$$ H(x) = \sum_{i=1}^I x^{\beta_i} P_i(\log x), \quad \beta_i \in \CC, \quad \alpha < \Re \beta_i \le 1, \quad P_i \text{ are polynomials}, $$
we have
$$ \sum_{n\le x} a_n = H(x) + \Omega(x^\alpha). $$
\end{lemma}

\proof
This is a simplified version of \cite[Th. 2]{kuhleitner1994}.
\endproof

\section{Asymptotic properties of \texorpdfstring{$\sum \taur(n)$}{Σ τr(n)}}

\begin{lemma}
Let $F_r(s)$ be Dirichlet series for $\taur$:
$$ F_r(s) := \sum_{n=1}^\infty {\taur(n)\over n^s}. $$
Then
\begin{equation}\label{eq:dirichlet-taur}
F_r(s) = {\zeta(s) \zeta\bigl((r+1)s\bigr) \over \zeta\bigl((2r+2)s\bigr)}, \qquad \sigma > 1.
\end{equation}
\end{lemma}

\proof
Let us transform Bell series for $\taur$:
\begin{multline*}
\taur_p(x) = \sum_{k=0}^\infty \taur(p^k) x^k
= \sum_{k=0}^r x^k + 2 \sum_{k>r} x^k
= \sum_{k=0}^\infty x^k + \sum_{k>r} x^k
=\\
= (1+x^{r+1}) \sum_{k=0}^\infty x^k
= {1+x^{r+1} \over 1-x}
= {1-x^{2r+2} \over (1-x)(1-x^{r+1})}.
\end{multline*}
The representation of $F_r$ in the form of an infinite product by~$p$ completes the proof:
$$
F_r(s) = \prod_p \taur_{p}(p^{-s}) = \prod_p {1-p^{-(2r+2)s} \over (1-p^{-s}) (1-p^{-(r+1)s})} = {\zeta(s) \zeta\bigl((r+1)s\bigr) \over \zeta\bigl((2r+2)s\bigr)}.
$$
\endproof

It follows from  \eqref{eq:dirichlet-taur} that
\begin{equation}\label{eq:taur-star-representation}
\taur = \tau(1, r+1; \cdot) \star \mu_{2r+2}
\end{equation}

\begin{theorem}\label{th:sum-of-taur-asymptotic}
If $\Delta$ is estimated as in \eqref{eq:Delta-theta-theta'} then for $r>0$
\begin{equation*}
\label{eq:sum-of-taur-asymptotic}
\sum_{n\le x} \taur(n) = Ax + Bx^{1/(r+1)} + \E_{r+1}(x),
\quad
\E_{r}(x) = O \left( x^{\max(\theta_{r}, 1/2r)} \log^{\theta'_{r}} x \right),
\end{equation*}
where constants $A$ and $B$ are specified below in \eqref{eq:sum-of-taur-asymptotic-with-const}.
\end{theorem}

\proof
Taking into account \eqref{eq:taur-star-representation} we have for $r>0$
\begin{multline*}
\sum_{n\le x} \taur(n)
= \sum_{n\le x} \mu_{2r+2}(n) T(1,r+1; x/n)
= \zeta(r+1) x \sum_{n\le x} {\mu_{2r+2}(n) \over n}
+ \\
+ \zeta\bigl(1/(r+1)\bigr) x^{1/(r+1)} \sum_{n\le x} {\mu_{2r+2}(n) \over n^{1/(r+1)}}
+ \sum_{n\le x} \mu_{2r+2}(n) \Delta(1, r+1, x/n).
\end{multline*}
But for $s \ge 1/k$
$$
\sum_{n\le x} {\mu_k(n) \over n^s}
= {1\over\zeta\bigl(ks\bigr)} - \sum_{n>x} {\mu_k(n) \over n^s}
= {1\over\zeta\bigl(ks\bigr)} + O(x^{1/k-s})
$$
and
\begin{multline*}
\sum_{n\le x} \mu_{2k}(n) \Delta(1, k, x/n)
= \sum_{n\le x^{1/2k}} \mu(n) \Delta(1, k, x/n^{2k})
\ll \\
\ll \sum_{n\le x^{1/2k}} \left( x \over n^{2k} \right)^{\theta_k} \log^{\theta'_k} x
\ll x^{\theta_k} \log^{\theta'_k} x \left( 1 + x^{1/2k - \theta_k} \right)
\ll x^{\max(\theta_k, 1/2k)}  \log^{\theta'_k} x.
\end{multline*}
So
\begin{equation}
\label{eq:sum-of-taur-asymptotic-with-const}
\sum_{n\le x} \taur(n)
= {\zeta(r+1) \over \zeta(2r+2)} x
+ {\zeta\bigl({1 \over r+1}\bigr) \over \zeta(2)} x^{1\over r+1}
+ O \left( x^{\max\bigl(\theta_{r+1}, 1/(2r+2)\bigr)} \log^{\theta'_{r+1}} x \right) .
\end{equation}
\endproof

For the case $r=0$ see \eqref{eq:gioia1966} above.

\begin{lemma}\label{l:a-la-delange}
Let $r>0$, $x^\eps \le y \le x^{1/2r}$. Then under RH we have
\begin{equation}\label{eq:taur-rh-error-term}
\E_r(x) = \sum_{n\le y} \mu(n) \Delta(1,r,x/n^{2r}) + O(x^{1/2+\eps}y^{1/2-r}+x^\eps).
\end{equation}
\end{lemma}

\proof
We follow the approach of Montgomery and Vaughan (see \cite{montgomery1981} or~\cite{cao2010}).

\medskip

First of all consider
$$ g_y(s) = {1\over \zeta(s)} - \sum_{d\le y} {\mu(d)\over d^s}. $$
Then for $\sigma>1$
$$ g_y(s) = \sum_{d>y} {\mu(d)\over d^s}. $$
Assuming RH we have by \cite[Th. 14.25]{titchmarsh1986}
$$ \sum_{d\le y} {\mu(d)\over d^s} = \zeta^{-1}(s) + O\bigl( y^{1/2-\sigma+\eps} (|t|^\eps+1) \bigr) \qquad \text{~for~} \sigma>1/2+\eps,$$
so
\begin{equation}\label{eq:gy_s}
g_y(s) \ll y^{1/2-\sigma+\eps} (|t|^\eps+1) \qquad \text{~for~} \sigma>1/2+\eps.
\end{equation}

Now let us split $\sum_{n\le x} \tau^{(r-1)}(n)$ into two parts:
$$ \sum_{n\le x} \tau^{(r-1)}(n)
= \sum_{d^{2r}\le x} \mu(d) T(1,r; x/d^{2r}) = S_1+S_2,$$
where
\begin{multline*}
S_1 := \sum_{d\le y} \mu(d) T(1,r; x/d^{2r})
= \zeta(r) x \sum_{d\le y} {\mu(d) \over d^{2r}} + \zeta(1/r) x^{1/r} \sum_{d\le y} {\mu(d)\over d^2} + \\ + \sum_{d\le y} \mu(d) \Delta(1,r; x/d^{2r})
\end{multline*}
and $S_2$ is the rest of $\sum_{n\le x} \tau^{(r-1)}(n)$.
We note that under RH by taking into account~$y\le x^{1/2r}$ we have
$$ x^{1/r} \sum_{d>y} {\mu(d)\over d^2} \ll x^{1/r} y^{-3/2+\eps} \ll x^{1/2}y^{1/2-r+\eps} $$
and so
$$ x^{1/r} \sum_{d\le y} {\mu(d)\over d^2} = {x^{1/r}\over\zeta(2)} + O(x^{1/2}y^{1/2-r+\eps}). $$

Next, let
$$h_y(s):= \zeta(s) \zeta(rs) g_y(2rs) x^s s^{-1}.$$
Then by Perron formula with $c=1+\eps$, $T=x^2$ one can estimate
$$ S_2 = {1\over 2\pi i} \int_{1+\eps-ix^2}^{1+\eps+ix^2} h_y(s) ds + O(x^\eps). $$
By moving line of integration to $[1/2+\eps-ix^2, 1/2+\eps+ix^2]$ we obtain
$$ S_2 = \res_{s=1} h(s) + O(I_1+I_2+I_3), $$
where
$$
I_1 = \int_{1  +\eps-ix^2}^{1/2+\eps-ix^2} h(s)ds, \qquad
I_2 = \int_{1/2+\eps-ix^2}^{1/2+\eps+ix^2} h(s)ds, \qquad
I_3 = \int_{1/2+\eps-ix^2}^{1  +\eps-ix^2} h(s)ds.
$$
Due to \eqref{eq:gy_s} and estimates of $\zeta$ under RH we have
\begin{align*}
g_y(2rs) \ll y^{1/2-r} (|t|^\eps+1) & \qquad \text{~for~} \sigma>1/2+\eps,
\\
h(s) \ll y^{1/2-r} (|t|^\eps+1) x^s s^{-1} & \qquad \text{~for~} \sigma>1/2+\eps,
\end{align*}
and
$$ I_{1,3} \ll y^{1/2-r+\eps} \max_{\sigma\in[1/2+\eps, 1+\eps]} x^{\sigma-2} \ll y^{1/2-r+\eps}, $$
$$ I_2 \ll y^{1/2-r+\eps} \int_1^{x^2} x^{1/2} t^{-1}dt
\ll y^{1/2-r+\eps} x^{1/2+\eps}.$$
Identity
$$ \res_{s=1} h(s) = \zeta(r) x \sum_{d>y} {\mu(d)\over d^{2r}} $$
completes the proof.
\endproof

\begin{theorem}
If $\Delta$ is estimated as in \eqref{eq:Delta-theta-theta'} and $\theta_{r} < 1/2r$ then under RH
$$ \E_{r}(x) = O(x^{\alpha}), \qquad \alpha = {1-\theta_r \over 2r+1-4r\theta_r}. $$
\end{theorem}

\proof
Let us start with \eqref{eq:taur-rh-error-term}:
\begin{multline*}
\E_r(x) = \sum_{n\le y} \mu(n) \Delta(1,r, x/n^{2r}) + O(x^{1/2+\eps} y^{1/2-r} + x^\eps)
\ll \sum_{n\le y} \left( x\over n^{2r} \right)^{\theta_r+\eps}
+ \\ +
O(x^{1/2+\eps}y^{1/2-r}+x^\eps)
\ll
x^\eps \left( x^{\theta_r} \left( 1+y^{1-2r\theta_r} \right) + x^{1/2}y^{1/2-r} + 1 \right).
\end{multline*}
If $\theta_{r} < 1/2r$ then
$$
\E_r(x) \ll x^{\eps} \left( x^{\theta_r} y^{1-2r\theta_r} + x^{1/2}y^{1/2-r} \right).
$$
Choice $y = x^{\beta}$, where
$$\beta = {1-2\theta_r \over 2r+1-4r\theta_r},$$
accomplishes the proof.
\endproof

For the values of $\theta_b$ from Table \ref{tbl:tauab} we have
$$
\max\bigl(\theta_{r}, 1/2r\bigr) = \begin{cases}
                               1/2r,     & r\le 2, \\
                               \theta_r, & r>2.
                               \end{cases}
$$
So currently the only non-trivial case of the previous theorem is an estimation for~$\tau^{(1)} \equiv \taues$. We get under assumption of RH that
$$ \sum_{n\le x} \tau^{(1)}(n) = {\zeta(2)\over \zeta(4)} x + {\zeta(1/2) \over \zeta(2)} x^{1/2} + O(x^{\alpha+\eps}), $$
where
$$ \alpha = {1-\theta_{2} \over 5-8\theta_{2}} = {3728\over 15469} \approx 0.241 < 1/4.$$

\begin{theorem}
\begin{equation}
\label{eq:sum-of-taur-asymptotic-Omega}
\E_{r}(x) = \Omega\left(x^{1/(2r+2)}\right).
\end{equation}
\end{theorem}

\proof
Equation \eqref{eq:sum-of-taur-asymptotic-Omega} is implied by the substitution
$m_1 = 1$, $ m_2 = r$, $n_1=2r$
into Lemma \ref{l:kuhleitner1994}. The choice of parameters plainly follows from~\eqref{eq:dirichlet-taur}. We obtain
$$\alpha = {1\over 2r+2}, $$
which is an exponent in the required $\Omega$-term.
\endproof

\section{Asymptotic properties of \texorpdfstring{$\sum \sigmar$}{Σ σr(n)}}

\begin{lemma}
Let $G_r(s)$ be Dirichlet series for $\sigmar$:
$$ G_r(s) := \sum_{n=1}^\infty {\sigmar(n)\over n^s}. $$
Then
\begin{equation}\label{eq:dirichlet-sigmaes}
G_r(s) = {\zeta(s-1) \zeta\bigl((r+1)s-r\bigr) \over \zeta\bigl((r+2)s-r-1\bigr)} H_r(s), \qquad \sigma > 2,
\end{equation}
where Dirichlet series $H_r(s)$ converges absolutely for $\sigma > (2r+2)/(2r+3)$.
\end{lemma}

\begin{proof}
Consider Bell series for $\sigmar$:
\begin{multline*}
\sigmar_p(x)
:= \sum_{k=0}^\infty \sigmar(p^k) x^k
= \sum_{k=0}^r p^kx^k + \sum_{k>r} (p^r+p^k)x^k
= \sum_{k=0}^\infty p^kx^k + \sum_{k>r} p^rx^k
=\\
= {1\over 1-px} + {p^rx^{r+1} \over 1-x}.
\end{multline*}
Then
$$
(1-px) \sigmar_p(x) = 1 + {p^rx^{r+1}(1-px) \over 1-x}
= 1 + \sum_{k=0}^\infty (p^rx^{r+1+k}-p^{r+1}x^{r+2+k})
$$
and
$$
{(1-px)(1-p^rx^{r+1})\over 1-p^{r+1}x^{r+2}} \sigmaes_p(x)
= 1 + {p^rx^{r+2}(1-px)(1-p^rx^r) \over (1-x)(1-p^{r+1}x^{r+2})}
:= h_p(x).
$$
For $\sigma > 1 $ we have
$$ h_p(p^{-s}) \ll p^{-2}. $$
For $1 \ge \sigma \ge (2r+2)/(2r+3) + \eps$ we have
$$
h_p(p^{-s}) \ll p^{2r+1-(2r+3)s} \ll p^{-1-\eps}.
$$
Now \eqref{eq:dirichlet-sigmaes} follows from the representation of $G_r$ in the form of infinite product by~$p$:
$$
G_r(s) = \prod_p \sigmar_{p}(p^{-s}).
$$

\end{proof}

Following theorem generalizes \eqref{eq:sitaramachandrarao1973}.

\begin{theorem}
$$ \sum_{n\le x} \sigmar(n) = Dx^2 + O(x\log^{5/3} x), \qquad  D = {\zeta(r+2) H_r(2) \over 2\zeta(r+3)}. $$
\end{theorem}

\proof
For a fixed $r$ let $z(n)$ be the coefficient at $n^{-s}$ of the Dirichlet series
$$
{\zeta(s-1) \zeta\bigl((r+1)s-r\bigr) \over \zeta\bigl((r+2)s-r-1\bigr)}
$$
and let $h(n)$ be the coefficient  of the Dirichlet series $H_r(s)$. It follows from \eqref{eq:dirichlet-sigmaes} that
$ \sigmar = z \star h $.
One can verify that
$$ z(n) = \sum_{ab^{r+1}c^{r+2}=n} ab^rc^{r+1}\mu(c). $$
Taking into account Lemma \ref{l:petermann1997} with $(\alpha,\beta) = (r+1,r)$ we obtain
\begin{multline*}
\sum_{n\le x} z(n) = \sum_{c \le x^{1/(r+2)}} c^{r+1} \mu(c) \left(
{\zeta(r+2) \over 2} {x^2\over c^{2r+4}} + O\bigl(x c^{-r-2} \log^{2/3} x\bigr)
\right)
=\\
= {\zeta(r+2) \over 2\zeta(r+3)} x^2 + O(x \log^{5/3} x)
\end{multline*}
because
$$ \sum_{c \le x^{1/(r+2)}} {\mu(c) \over c^{r+3}} = {1\over\zeta(r+3)} - \sum_{c > x^{1/(r+2)}} {\mu(c) \over c^{r+3}} = {1\over\zeta(r+3)} + O(x^{-1}) $$
and
$$ \sum_{c \le x^{1/(r+2)}} {\mu(c) \over c} \ll \sum_{c\le x} {1\over c} \ll \log x. $$
Now
\begin{multline*}
\sum_{n\le x} \sigmar(n)
= \sum_{n\le x} h(n) \left( {\zeta(r+2) \over 2\zeta(r+3)} {x^2\over n^2} + O\left({x\over n} \log^{5/3} x\right) \right)
=\\
= {\zeta(r+2) \over 2\zeta(r+3)} x^2 \sum_{n\le x} {h(n)\over n^2}
+ O\biggl( x \log^{5/3} x \sum_{n\le x} {h(n) \over n} \biggr).
\end{multline*}
But $H_r(s)$ converges absolutely at $\sigma \ge (2r+2)/(2r+3)+\eps$, so
$$ \sum_{n\le x} {h(n) \over n} \ll O(1) $$
and
$$ \sum_{n\le x} {h(n)\over n^2} = H_r(2) - \sum_{n>x} {h(n)\over n^2}
= H_r(2) + O(x^{-(2r+4)/(2r+3)+\eps}). $$
\endproof

\begin{theorem}
For a fixed $r>0$
$$ \sum_{n\le x} \sigmar(n) = Dx^2 + \Omega_{\pm} (x\llog x). $$
\end{theorem}

\proof
The proof almost replicates the proof of \cite[Th. 3]{petermann1997} with following changes (in notations of \cite{petermann1997}):
$$ \kappa(n) := {\sigmar(n) \over n}, $$
$$ \sum_{n=1}^\infty {\kappa(n)\over n^s} = {\zeta(s) \zeta\bigl((r+1)s+1\bigr) \over \zeta\bigl((r+2)s+1\bigr)} H_r(s+1), $$
$$ \upsilon := \mu \star \kappa, $$
$$ \sum_{n=1}^\infty {\ups(n)\over n^s} = { \zeta\bigl((r+1)s+1\bigr) \over \zeta\bigl((r+2)s+1\bigr)} H_r(s+1), $$
$$ \ups(p^a) = {\sigmar(p^a)\over p^a} - {\sigmar(p^{a-1})\over p^{a-1}} = \begin{cases}
0, & a \le r+1, \\
1/p, & a=r+1, \\
p^{r-a}-p^{r-a+1}, & a>r+1.
\end{cases} $$
We take $m:=\log^{1/(4r+4)} x$ and
$$ A := \prod_{p\le m} p^{r+1} \sim e^{(r+1)m} \sim \exp(\log^{1/4}x) , $$
then
$$ G = \sum_{k\le u(x)} {\ups(k)\over k} \gcd(A,k)
= \sum_{n^{r+1} \mid A} \ups(n^{r+1}) \suma_{k\le u(x)/n^{r+1}} {\ups(n^{r+1}k)\over \ups(n^{r+1}) k }.$$
Here $\sum^*_k$ means summation over $k$ such that for every $p \mid k$ we have~$p \mid n$ or $p\nmid A$.
Taking into account $\ups(p^{r+1})=1/p$ we get
\begin{multline*}
G = \sum_{n^{r+1} \mid A} \ups(n^{r+1}) \suma_{k\ge1} {\ups(n^{r+1}k)\over \ups(n^{r+1}) k} + o(1)
=\\
= \sum_{n^{r+1} \mid A} \ups(n^{r+1}) \prod_{p \mid n} \left( 1+\sum_{\nu\ge r+2} {\ups(p^\nu)\over p^{\nu-r-2}} \right) \prod_{p>m} \left( 1+\sum_{\nu\ge r+1} {\ups(p^\nu) \over p^\nu}\right) + o(1).
\end{multline*}
Since $|\ups(p^\nu)| \le 1/p$ we obtain
$$ \sum_{\nu\ge r+1} {\ups(p^\nu) \over p^\nu} \ll p^{-r-2}. $$
Since $\ups(n^{r+1}) = 1/n$ for $n^{r+1} \mid A$ and $\log m\asymp\llog x$ we have
\begin{multline*}
G = \bigl(1+o(1)\bigr) \sum_{n^{r+1} \mid A} {1\over n} \prod_{p \mid n} \left( 1+\sum_{\nu\ge r+2} {\ups(p^\nu)\over p^{\nu-r-2}} \right) =
\\=
\bigl(1+o(1)\bigr) \prod_{p\le m} \left( 1+{1\over p} + \sum_{\nu\ge r+2} {\ups(p^\nu) \over p^{\nu-r-1}} \right).
\end{multline*}
But $\ups(p^\nu) \ge 1/2p^{\nu-r-1}$ for $a\ge r+2$. So
$$ \sum_{\nu\ge r+2} {\ups(p^\nu) \over p^{\nu-r-1}} \ge \sum_{\nu\ge r+2} {1\over 2(p^{\nu-r-1})^2} \ge {1\over 2p^2}. $$
Hence
$$ G \ge \bigl(1+o(1)\bigr) \prod_{p\le m} (1+p^{-1}+p^{-2}/2) \gg \prod_{p\le m} (1+p^{-1}) \gg \log m \gg \llog x. $$
\endproof

\section{Some remarks}

The estimate \eqref{eq:sup-taur} implies that $ { \taur(n) / n } \to0$ as $n\to\infty $. Thus it is natural to ask what is the maximum value of this ratio.

\begin{lemma}
For $n\ge1$ we have
$$ \taur(n) \le n, $$
where the equality has a place only if $n=1$ or if $n=2$ and $r=0$.
\end{lemma}

\proof
Recalling the definition \eqref{eq:taur-def}
we obtain that the least value of $a$ for which~$\taur(p^a) $ is different from~$1$ is $a=r+1$. So
$$ \taur(n) = 2^{\#\{p^{r+1} \mid n\}} \le 2^{(\log_2 n) / (r+1)} = n^{1/(r+1)} $$
and the statement of the lemma easily follows.
\endproof

\medskip

One can see that \eqref{eq:sup-sigmar} implies
$$ {\sigmar(n) \over n} \to +\infty, \qquad n\to\infty.$$

\begin{theorem}
Consider the distribution function
$$ S_N(q,r; \lambda) := {1\over N} \#\{ n\le N \mid \sigmar(n^q) \le \lambda n^q \}, \qquad q,r \in\N. $$
Then $S_N(q,r; \lambda)$ weakly converges to a function $S(q,r; \lambda)$ which is continuous if and only if~$q>r$.
\end{theorem}

\proof
Let us fix arbitrary $q$ and $r$ and let
$$ f(n) := \ln {\sigmar(n^q) \over n^q}, $$
here $f$ is an additive function. It is enough to prove that
$$ F_N(\lambda) := {1\over N} \#\{ n\le N \mid f(n)\le\lambda \} $$
converges weakly to some $F(\lambda)$ as $N\to\infty$ and $F$ is continuous if and only if $q>r$.

By definition \eqref{eq:sigmar-def}
$$
\sigmar(p^q) = \begin{cases}
               p^q, & r\ge q, \\
               p^q+p^r, & r<q.
               \end{cases}
$$
So
$$
f(p) = \begin{cases}
       0, & r\ge q, \\
       \ln(1+p^{r-q}) \ll p^{r-q}, & r<q,
       \end{cases}
$$
and $f(p) = |f(p)| \le 1$. Also
$$
\sum_{p} {f(p)\over p} \ll \begin{cases}
       0, & r\ge q, \\
       \sum_p p^{r-q-1} \ll \sum_p p^{-2}, & r<q,
       \end{cases}
\quad < +\infty.
$$
and the same is valid for $\sum_p f^2(p)/p$. Now by Erdős---Wintner theorem~\cite[Th. i]{erdos1939} we get that $F_N(\lambda)$ converges weakly to  $F(\lambda)$ as $N\to\infty$. Taking into account
$$
\sum_{f(p)\ne0} {1\over p} \ll \begin{cases}
       0, & r\ge q, \\
       \infty, & r<q,
       \end{cases}
$$
we see that due to \cite[Th. v]{erdos1939} the distribution $F$ is continuous if and only if $r<q$.
\endproof

\bibliographystyle{ugost2008s}
\bibliography{taue}

\end{document}